\documentclass[11pt]{amsart}
\usepackage[colorlinks=true,pagebackref,hyperindex,citecolor=green,linkcolor=blue]{hyperref}
\usepackage{amsfonts,amssymb,amsthm,amsmath,amsxtra,amscd,verbatim,eucal}
\usepackage[all]{xy}
\usepackage{amsmath}
\usepackage{amsfonts}
\usepackage{amssymb}
\usepackage{color}
\usepackage{fullpage}
\usepackage{mathrsfs}
\usepackage{setspace}
\usepackage{moreverb}
\usepackage{tikz}
\usetikzlibrary{matrix}
\usepackage[T1]{fontenc}
\usepackage[english]{babel}

\DeclareUnicodeCharacter{0301}{\`{A}}

\newcommand{\ignore}[1]{}

\newtheorem{theodef}{Theorem/Definition}[section]
\newtheorem{theorem}{Theorem}[section]

\newtheorem{corollary}[theorem]{Corollary}
\newtheorem{proposition}[theorem]{Proposition}

\theoremstyle{definition}
\newtheorem{definition}[theorem]{Definition}
\newtheorem{example}[theorem]{Example}

\theoremstyle{remark}
\newtheorem{remark}[theorem]{Remark}

\numberwithin{equation}{section}

\input xy
\xyoption{all}

\newcommand{\KK}{\mathbb{K}}

\newcommand{\ZZ}{{\mathbb{Z}}}

\newcommand{\Ass}{\mbox{\rm{Ass}} }

\newcommand{\dm}{\mathrm {dim }}

\newcommand{\Ext}{\mbox{\rm{Ext}} }

\newcommand{\fM}{{\mathfrak m}}

\newcommand{\la}{{\lambda}}

\begin{document}

\title[Lyubeznik table of $S_r$ and $CM_t$  rings]{Lyubeznik table of $S_r$ and $CM_r$ rings}

\author[J. \`{A}lvarez Montaner ]{Josep \`{A}lvarez Montaner}
\address{Departament de Matem\`atiques  and  Institut de Matem\`atiques de la UPC-BarcelonaTech (IMTech)\\  Universitat Polit\`ecnica de Catalunya \\ Av.~Diagonal 647, Barcelona
08028; and Centre de Recerca Matem\`atica (CRM), 08193 Bellaterra, Barcelona.} 
\email{josep.alvarez@upc.edu}
\thanks{ JAM is partially supported by grant  PID2019-103849GB-I00 (MCIN/AEI/10.13039/501100011033), AGAUR grant 2021 SGR 00603 and Spanish State Research Agency, through the Severo Ochoa and Maria de Maeztu Program for Centers and Units of Excellence in R$\&$D (project CEX2020-001084-M)}

\author[S. Yassemi]{Siamak Yassemi}
\address{Department of Mathematics, Purdue University, West Lafayette,
IN 47907} 
\email{syassemi@purdue.edu}

%
%
%
%
%
%
%
%
%
\begin{abstract}
We describe the shape of the Lyubeznik table of either rings in positive characteristic or Stanley-Reisner rings in any characteristic when they satisfy Serre's condition $S_r$ or they are Cohen-Macaulay in a given codimension, condition denoted by $CM_r$. Moreover we show that these results are sharp.
\end{abstract}

\maketitle

\section{Introduction}
Let $(R,\fM)$  be a regular local ring containing a field $\KK$ and set $A=R/I$, where $I$ is an ideal of $R$. It is known that some vanishing results on local cohomology modules behave similarly in either the case where $\KK$ is a field of positive characteristic or $I$ is a squarefree monomial ideal in polynomial ring in any characteristic. For example, as a consequence of work of Peskine and Szpiro \cite{PS} in positive characteristic and  Lyubeznik \cite{Lyu84} for  monomial ideals, there is only one local cohomology different from zero when $A$ is Cohen-Macaulay. 
The main reason behind this similar behaviour is that the Frobenius morphism in positive characteristic is flat by Kunz theorem \cite{Kunz} and, applying it to our ideal $I$ recursively gives us a cofinal system of ideals with respect to the system given by the usual powers which describe these local cohomology modules. For squarefree monomial ideals we have a similar flat morphism, raising any variable of the polynomial ring to its second power, that plays the same role. This point of view has already been successfully used by Singh and Walther \cite{AnuragUliPure} and  \`{A}lvarez Montaner \cite{AM15}. 

\vskip 2mm

The vanishing of local cohomology modules implies the vanishing of some Lyubeznik numbers of $A$ introduced in  \cite{Lyu93}. Indeed, using an spectral sequence argument
one may check that the Lyubeznik table of a Cohen-Macaulay ring $A$ is trivial \cite[Remark 4.2]{Alv00}. This still holds true replacing the Cohen-Macaulay property for sequentially Cohen-Macaulay \cite{AM15}. We point out that these results are no longer true when $A$ is Cohen-Macaulay containing a field of characteristic zero. Take for example the ideal generated 
by the $2\times 2$ minors of a generic $2\times 3$ matrix \cite{AL06}. 

\vskip 2mm

In this note we continue the study of Lyubeznik numbers of $A$ in either the case where $\KK$ is a field of positive characteristic or $I$ is a squarefree monomial ideal in any characteristic. The main results are Theorems \ref{thm1} and \ref{thm2} where we describe the shape of the Lyubeznik table of $A$ when we relax the Cohen-Macaulay condition on $A$ to Serre's  condition $S_r$ or being Cohen-Macaulay in codimension $r$, condition denoted by $CM_r$. 


\vskip 2mm

A priori, there is no reason for thinking that the results we obtain are sharp but this is indeed the case as shown in Example \ref{ex1}. Finally we highlight that, using results obtained by Conca and Varbaro \cite{CV20}, one may compute some apparently complicated Lyubeznik tables in positive characteristic in the event that the ring $A$ has a squarefree Gr\"obner deformation.

\section{Lyubeznik numbers}

Let $(R,\fM)$  be a regular local ring containing a field $\KK$ and $I$ an ideal of $R$. Some finiteness properties
of local cohomology modules $H^r_I(R)$ were proved by 
Huneke and Sharp  \cite{HunekeSharp} when the field $\KK$ has positive characteristic 
and Lyubeznik \cite{Lyu93}  in the characteristic zero case. 
In particular, they proved that Bass numbers of these local cohomology modules are finite.  
Relying on this fact, Lyubeznik \cite{Lyu93} introduced a set of numerical invariants of local rings containing a field as follows:

\begin{theodef}
Let $A$ be a local ring containing a field $\KK$, so that its completion $\widehat{A}$ admits a surjective
ring homomorphism $\xymatrix@1{R\ar@{->>}[r]^-{\pi}& \widehat{A}}$ from a regular local ring
$(R,\mathfrak{m})$ of dimension $n$ and set $I:=\ker (\pi)$. Then, the Bass numbers 
\[
\lambda_{p,i} (A):=\mu_p (\mathfrak{m},H_I^{n-i}(R))=\mu_0 (\mathfrak{m},H_{\mathfrak{m}}^p (H_I^{n-i}(R)))
\]
depend only on $A$, $i$ and $p$, but neither on $R$ nor on $\pi$.
\end{theodef}

We refer to these invariants as \emph{Lyubeznik numbers} and they are known to
 satisfy the following properties: $\la_{p,i}(A)=0$ if $i>d$,  $\la_{p,i}(A)=0$ if $p>i$ and $\la_{d,d}(A)\neq 0$,
 where $d=\dm A$. Therefore we can collect them in
the so-called {\it Lyubeznik table}:
$$\Lambda(A)  = \left(
                    \begin{array}{ccc}
                      \la_{0,0} & \cdots & \la_{0,d}  \\
                       & \ddots & \vdots \\
                       &  & \la_{d,d} \\
                    \end{array}
                  \right).
$$

We say that the Lyubeznik table is {\it trivial} if $\la_{d,d}(A)=1$ and $\la_{p,i}(A)=0$ for $p$ and $i$ different from $d$.
The {\it highest} Lyubeznik number  $\la_{d,d}(A)$ has an interpretation in terms of the dual graph $\Gamma_1(A)$, also known as Hochster-Huneke graph, associated to ${\it Spec}(A)$.

\begin{definition}
Let $A$ be a ring of dimension $d$ and let $t$ be an integer such that $0\leq t \leq d$. We define the graph $\Gamma_t(A)$ as a simple graph whose vertices are the minimal primes of $A$ and there is an edge between $\mathfrak{p}$ and $\mathfrak{q}$ distinct minimal primes if and only if 
$ht(\mathfrak{p} +\mathfrak{q}) \leq t$.
\end{definition}

Zhang gave the following characterization.

\begin{theorem}\cite[Main Theorem]{Zha07} 
Let $A$ be a complete local ring with separably closed residue field. Then:
$$\lambda_{d,d}(A) = \# \Gamma_1(A)$$ 
\end{theorem}

\begin{remark}
More generally $\lambda_{d,d}(A) = \# \Gamma_1(B)$ where $B=\widehat{\widehat{A}^{\rm sh}}$ is the completion of the strict henselianization of the completion of $A$.
\end{remark}

We point out that Kawasaki already proved in \cite[Theorem 2]{Kaw02}  that the highest Lyubeznik number $\lambda_{d,d}$ 
of a Cohen-Macaulay ring (or even $S_2$) is always one. Other Lyubeznik numbers can be described from the graphs $\Gamma_t(A)$  as shown by Walther \cite{Wal01} and  N\'u\~nez-Betancourt, Spiroff and Witt \cite{NBSW19}. Moreover, Walther describe the possible Lyubeznik tables for $d\leq 2$ (see also \cite{RWZ22} for other small dimensional cases).

\begin{proposition}
Let $A$ be an equidimensional complete local ring  of dimension $\geq 3$ with separably closed residue field. Then

\begin{itemize}
\item[(i)] \cite[Proposition 2.2]{Wal01} $\lambda_{0,1} (A)= \# \Gamma_{d-1}(A) -1$.

\item[(ii)] \cite[Theorem 5.4 (1)]{NBSW19} $\lambda_{1,2} (A) = \# \Gamma_{d-2}(A) - \# \Gamma_{d-1}(A)$.

\item[(iii)] \cite[Theorem 5.4 (2)]{NBSW19} $\lambda_{i,i+1} (A) \geq \# \Gamma_{d-i-1}(A) - \# \Gamma_{d-i}(A)$ for $1\leq i \leq d-2$.
\end{itemize}

\end{proposition}

\section{Lyubeznik tables of $S_r$ and $CM_r$ rings}

Throughout this section we will always assume that $(R,\mathfrak{m})$  is a regular local ring and 
$A$ is  a  complete local ring containing a field that admits a presentation $A= R/I$ where $I \subseteq R$ is an ideal. 
We will study the Lyubeznik table when we relax the Cohen-Macaulay condition on the ring $A$. The classical way of doing so is by means of Serre's conditions. Another way is by asking for being Cohen-Macaulay up to some codimension. This notion has been considered by Miller, Novik and Schwarz \cite{MNS11}. It was further developed, in the case that $A$ is equidimensional and defined by a squarefree monomial ideal, in \cite{HYZN12, HYZN12, PPTY22}.

\begin{definition}
We say:
\begin{itemize}
\item[(i)] $A$ satisfies Serre's condition $S_r$ if 
$${\rm depth} \; A_{\mathfrak{p}} \geq {\rm min} \{ r, \dim A_{\mathfrak{p}}\},$$ for all $ \mathfrak{p} \in {\rm Spec}(R)$.

\vskip 2mm

\item[(ii)] $A$  satisfies the condition $CM_r$ if it  is Cohen-Macaulay in codimension $r$, that is $A_{\mathfrak{p}}$ is Cohen-Macaulay for all $ \mathfrak{p} \in {\rm Spec}(R)$ with ht $\mathfrak{p}\leq d-r$.
\end{itemize}
\end{definition}


\begin{remark}
Schenzel \cite{Sch79} proved that if $A$ satisfies $S_r$ with $r\geq 2$ then it is equidimensional.  However, we may have non-equidimensional $CM_r$ rings (see Example \ref{ex1}). 
\end{remark}

Both the $S_r$ and $CM_r$ conditions can be characterized in terms of the {\it deficiency modules}
$$K^i_A:=\Ext^{n-i}_R(A,R).$$
The following result can be found  in the the work of Schenzel \cite[Lemma 3.2.1]{Sch82}  (see also \cite[Remark 2.9]{CV20}). For the squarefree monomial ideals case one may consult \cite{PPTY22}.

\begin{proposition} We have:

\begin{itemize}
\item[(i)] $A$ is $S_r$, $r\geq 2$, if and only if $\dim K^{i}_A \leq i-r$ for all $1\leq i \leq d$.
\item[(ii)] $A$ is $CM_r$ if and only if $\dim K^{i}_A \leq r$ for all $1\leq i \leq d$.
\end{itemize}
\end{proposition}

\vskip 2mm

Next we present the main results of the paper where the shape of the Lyubeznik tables is given in terms of the $S_r$ and the $CM_r$ conditions.

\vskip 2mm

\begin{theorem} \label{thm1}
 Assume that $r\geq 2$ and either that:
\begin{itemize}
\item[$\bullet$] $A$ is $S_r$ and contains a field of positive characteristic, or
\item[$\bullet$] $A$ is $S_r$ and $I$ is a squarefree monomial ideal.
\end{itemize}
Then, the Lyubeznik table of $A$  satisfies
$\lambda_{i,i}= \lambda_{i,i+1}= \cdots =\lambda_{i,i+(r-1)}=0,$ for $ i\in\{0,\dots, d-1\}.$
\end{theorem}

\begin{proof}
If $A$ contains a field of positive characteristic, then  Huneke and Sharp \cite[Corollary 2.3]{HunekeSharp} proved  that ${\Ass} \,(H^{n-i}_I(R)) \subseteq {\Ass} \,(K^i_A)$, and thus ${\dim} \,(H^{n-i}_I(R)) \leq {\dim} \,(K^i_A)$. In the squarefree monomial ideal case, Yanagawa \cite[Theorem 2.11] {Yan01} proved that the straight module $H^{n-i}_I(R)$ is equivalent to the squarefree module $K^i_A$. In particular this gives the equality $\dim \,(H^{n-i}_I(R)) = {\dim} \,(K^i_A)$ \cite[Lemma 2.8]{Yan01}.

\vskip 2mm

Now assume in both cases that $A$ is $S_r$ and thus we have  $\dim \,(K^{i}_A) \leq i-r$ and consequently $\dim \,(H^{n-i}_I(R)) \leq i-r$ for all $1\leq i \leq d$. Then the result follows 
from the inequality $${\rm id}_R \,(H^{n-i}_I(R) ) \leq \dim \,(H^{n-i}_I(R)) $$ proved in   \cite[Corollary  3.9]{HunekeSharp} and \cite[Theorem 3.4]{Lyu93}.
Namely, the Lyubeznik numbers are the Bass numbers $\lambda_{p,i} (A)=\mu_p (\mathfrak{m},H_I^{n-i}(R))$ and thus the possible non-zero $\lambda_{p,i} (A)$ are in the range $0\leq p\leq i-r$. 
\end{proof}

\begin{theorem}  \label{thm2}
Assume either that:
\begin{itemize}
\item[$\bullet$] $A$ is $CM_r$ and contains a field of positive characteristic, or
\item[$\bullet$] $A$ is $CM_r$ and $I$ is a squarefree monomial ideal.
\end{itemize}
Then the Lyubeznik table of $A$  satisfies
$\lambda_{p,i}= 0 , \hskip 2mm \forall p\geq r $ and $i\in\{0,\dots, d-1\}. $
\end{theorem}

\begin{proof}
The proof is analogous to the proof of Theorem \ref{thm1} but in the present case we have
 $\dim \,(K^{i}_A) \leq r$  and thus $\dim \,(H^{n-i}_I(R)) \leq r$ for all $1\leq i \leq d$. 
\end{proof}

\begin{remark}
Under the hypothesis of Theorem \ref{thm2}, assume that $A$ is $CM_1$ and thus the only possible non-zero row of the Lyubeznik table is the $0$-th row. Then,  the Lyubeznik numbers of $A$  satisfy
$\lambda_{d,d}= \lambda_{0,1} +1$ and    $\lambda_{p,d}= \lambda_{0,d-p+1}$  for all $p\in\{2,\dots, d-1\} $ (see  \cite{GS98,BB05}). 
\end{remark}
\vskip 2mm

Using Grothendieck's spectral sequence $$E_2^{p,n-i}= H_{\fM}^p(H_I^{n-i}(R))\Longrightarrow H_{\fM}^{p+n-i}(R) $$ we can give a similar result for the $CM_2$ case.

\begin{corollary} \label{cor36}
Assume either that :
\begin{itemize}
\item[$\bullet$] $A$ is $CM_2$ and contains a field of positive characteristic, or
\item[$\bullet$] $A$ is $CM_2$ and $I$ is a squarefree monomial ideal.
\end{itemize}
Then the Lyubeznik numbers of $A$  satisfy
$\lambda_{d,d}= \lambda_{0,1} +  \lambda_{1,2} + 1$, $\lambda_{2,d}= \lambda_{0,d-1}$ and $\lambda_{p,d}= \lambda_{0,d-p+1}+\lambda_{1,d-p+2}$  for all $p\in\{3,\dots, d-1\}.$
\end{corollary}

\begin{proof}
Under the $CM_2$ condition, the only possibly non-zero terms of Grothendieck spectral sequence are placed at the dot spots in the following diagram: 

\vskip 2mm

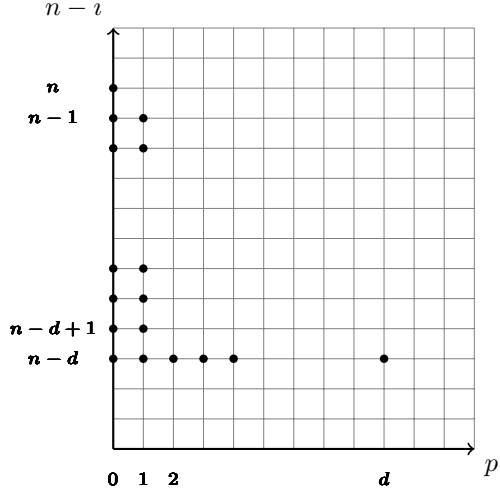
\begin{center}
\begin{tikzpicture}[scale=0.8]
\draw[step=.5cm,gray,very thin] (0,0) grid (6,7);
\draw[thick,->] (0,0) -- (6,0) node[anchor=north west] {{\small $p$}};
\draw[thick,->] (0,0) -- (0,7) node[anchor=south east] {{\small $n-i$}};
    \coordinate (1) at (1,1.5);
     \coordinate (11) at (0,1.5);
      \coordinate (12) at (0.5,1.5);
        \coordinate (2) at (1.5,1.5);
            \coordinate (3) at (2,1.5);
        \coordinate (4) at (4.5,1.5);
          \coordinate (5) at (0,2);
        \coordinate (6) at (0,2.5);
            \coordinate (7) at (0,3);
        \coordinate (8) at (0,5.5);
         \coordinate (9) at (0,6);
        \coordinate (10) at (0,5);
            \coordinate (15) at (0.5,2);
        \coordinate (16) at (0.5,2.5);
            \coordinate (17) at (0.5,3);
        \coordinate (18) at (0.5,5.5);
         \coordinate (19) at (0.5,6);
        \coordinate (20) at (0.5,5);
     \foreach \x in {(1),(2),(3),(4),(5),(6),(7),(8),(9),(10),(11),(12),(15),(16),(17),(18),(20)}{
        \fill \x circle[radius=2pt];
         \node (27) at ( -1,1.5)  {{\tiny $n-d$}};
           \node (27) at ( -1,2)  {{\tiny $n-d+1$}};
           \node (27) at ( -1,5.5)  {{\tiny $n-1$}};
              \node (27) at ( -1,6)  {{\tiny $n$}};
              
                \node (27) at ( 0,-0.5)  {{\tiny $0$}};
           \node (27) at ( 0.5,-0.5)  {{\tiny $1$}};
            \node (27) at ( 1,-0.5)  {{\tiny $2$}};
              \node (27) at ( 4.5,-0.5)  {{\tiny $d$}};
        }
\end{tikzpicture} 
\end{center}

\vskip 2mm

We have $\lambda_{0,0}=0$ by Grothendieck's vanishing theorem (see \cite[Theorem 6.1.2]{BS}). 
We also notice that $\lambda_{0,d}=\lambda_{1,d}=0$.

\vskip 2mm

The only possible non-zero differentials at each $E_j$-page, $j\geq 2$, of the spectral sequence are:
$$d_j: E_j^{0,n-j+1} \longrightarrow E_j^{j,n-d}  \;\; {\rm and} \; \; d_j: E_j^{1,n-j+1} \longrightarrow E_j^{j+1,n-d}.$$
By the general theory of spectral sequences, there exist filtrations $0\subseteq F^r_{n} \subseteq \cdots \subseteq F^r_{0} \subseteq H^r_{\fM}(R)$ for all $r$, such that the consecutive quotients are $F^r_{i}/F^r_{i+1}= E_\infty^{i,r-i}$.
Then, taking into account that $H^r_{\fM}(R)=0$ for all $r\neq n$, we have first:
\begin{itemize}
\item[$\bullet$] $0= E_\infty^{0,n-d+1}=E_{3}^{0,n-d+1}= \ker\left( d_2: E_2^{0,n-d+1} \longrightarrow E_2^{2,n-d}\right),$
\item[$\bullet$] $0= E_\infty^{2,n-d}=E_{3}^{2,n-d}= E_{2}^{2,n-d} / {\rm Im}\left( d_2: E_2^{0,n-d+1} \longrightarrow E_2^{2,n-d}\right),$
\end{itemize}
and thus $\lambda_{2,d}= \lambda_{0,d-1}$. 
For the next subdiagonal in the diagram we have, in the third page:
\begin{itemize}
\item[$\bullet$] $ E_3^{0,n-d+2}=E_{2}^{0,n-d+2},$
\item[$\bullet$] $0= E_\infty^{1,n-d+1}=E_{3}^{0,n-d+1}= \ker\left( d_2: E_2^{1,n-d+1} \longrightarrow E_2^{3,n-d}\right),$
\item[$\bullet$] $E_{3}^{3,n-d}= E_{2}^{3,n-d} / {\rm Im}\left( d_2: E_2^{1,n-d+1} \longrightarrow E_2^{3,n-d}\right),$
\end{itemize}
and in the fourth page:
\begin{itemize}
\item[$\bullet$] $0= E_\infty^{0,n-d+2}=E_{4}^{0,n-d+2}= \ker\left( d_3: E_3^{0,n-d+2} \longrightarrow E_3^{3,n-d}\right),$
\item[$\bullet$] $0= E_\infty^{2,n-d}=E_{4}^{2,n-d}= E_{3}^{2,n-d} / {\rm Im}\left( d_3: E_3^{0,n-d+2} \longrightarrow E_3^{3,n-d}\right).$
\end{itemize}
Therefore $\lambda_{3,d}= \lambda_{0,d-2}+\lambda_{1,d-1}$ and analogously we get $\lambda_{p,d}= \lambda_{0,d-p+1}+\lambda_{1,d-p+2}$  for all $p\in\{4,\dots, d-1\}.$ 
For the last case we only have to put into the picture the fact that $H^n_{\fM}(R)$ is isomorphic to the injective hull of the residue field which accounts for the $+1$ in the formula $\lambda_{d,d}= \lambda_{0,1} +  \lambda_{1,2} + 1$.

\end{proof}

\section{Squarefree monomial ideals}

A way to interpret Lyubeznik numbers for the case of squarefree monomial ideals  is in terms of the linear strands of the free resolution of the Alexander dual of the ideal.  This approach was given by  \`{A}lvarez Montaner and Vahidi \cite{AMV14} (see also \cite{AMY18}) and we will briefly recall it here. Throughout this section we  let $R=\KK[x_1,\dots, x_n]$ be a polynomial ring  with coefficients in a field $\KK$. Bass numbers behave well with respect to localization and completion so there is no inconvenience in working in this setting. The aim of this section is to use these methods to prove that the results given in Theorems \ref{thm1} and \ref{thm2} are sharp.

\vskip 2mm

Let $I^\vee$ be the Alexander dual of a squarefree monomial ideal $I\subseteq R$. Its 
 minimal  {\it $\mathbb{Z}$-graded free resolution} is an
exact sequence of free ${\mathbb{Z}}$-graded $R$-modules:
$$\mathbb{L}_{\bullet}(I^\vee): \hskip 3mm \xymatrix{ 0 \ar[r]& L_{m}
\ar[r]^{d_{m}}& \cdots \ar[r]& L_1 \ar[r]^{d_1}& L_{0} \ar[r]& I^\vee
\ar[r]& 0}$$  where the $j$-th term is of the form $$L_j =
\bigoplus_{\ell \in {\mathbb{Z}}}
R(-\ell)^{\beta_{j,\ell}(I^\vee)},$$ and the matrices of the
morphisms $d_j: L_j\longrightarrow L_{j-1}$ do not contain
invertible elements.  The $\mathbb{Z}$-graded {\it Betti numbers} of $I^\vee$
are the invariants $\beta_{j,\ell}(I^\vee)$. Given an integer $r$, the
{\it $r$-linear strand} of $\mathbb{L}_{\bullet}(I^\vee)$ is the complex:
$$\mathbb{L}_{\bullet}^{<r>}(I^\vee): \hskip 3mm \xymatrix{ 0 \ar[r]&
L_{n-r}^{<r>} \ar[r]^{d_{n-r}^{<r>}}& \cdots \ar[r]& L_1^{<r>}
\ar[r]^{d_1^{<r>}}& L_{0}^{<r>} \ar[r]& 0},$$ where $$L_j^{<r>} =  R(-j -r)^{\beta_{j,j+r}(I^\vee)} ,$$ and the
differentials $d_j^{<r>}: L_j^{<r>}\longrightarrow L_{j-1}^{<r>}$
are the corresponding components of $d_j$. 

\vskip 2mm

We point out that these differentials can be described using the so-called {\it monomial
matrices} introduced by Miller \cite{Mil00}.  These are matrices with scalar entries that keep track
of the degrees of the generators of the summands in the source and
the target.  Now  we construct a complex of $\KK$-vector spaces
$$\mathbb{F}_{\bullet}^{<r>}(I^{\vee})^{\ast}: \hskip 3mm \xymatrix{ 0 &
{\underbrace{\KK^{\beta_{n-r,n}(I^\vee) }}_{\deg 0}}\ar[l]& \cdots \ar[l]& {\underbrace{\KK^{\beta_{1,1+r}(I^\vee)}}_{\deg n-r-1} }\ar[l]&
{\underbrace{\KK^{\beta_{0,r}(I^\vee)} }_{\deg n-r}} \ar[l]& 0 \ar[l]},$$
where the morphisms are given by the transpose of the corresponding monomial matrices and thus we reverse the indices of the complex.
Then, the Lyubeznik numbers are described by means of the homology groups of these complexes (see \cite[Cor. 4.2]{AMV14}). 
 \begin{theorem} 
 Let $I^\vee$ be the Alexander dual of a squarefree monomial ideal $I\subseteq R$. Then
 $$\lambda_{p,n-r}(R/I)= {\rm
dim}_{\KK} H_{p}(\mathbb{F}_{\bullet}^{<r>}(I^{\vee})^{\ast}).$$
\end{theorem}

\vskip 2mm

It has been shown in \cite{HSYZ18}, \cite{VZN19}, \cite{PPTY22} that the $S_r$ and $CM_r$ properties on the ring $R/I$ provide conditions on  the vanishing of Betti numbers of the Alexander dual ideals $I^{\vee}$ and consequently the shape of the corresponding {\it Betti table}.  In particular it describes the linear strands of the free resolution. To compute Lyubeznik numbers we have to take a step further and consider the homology of these linear strands so, a priori, it may seem that the results in Theorems \ref{thm1} and \ref{thm2} are not sharp. The following examples show that indeed the results are sharp.

\vskip 2mm

\begin{example}\label{ex4}

Let
$I=(x_1,x_2,x_3,x_4) \cap (x_1, x_2, x_4, x_6) \cap (x_1,x_2,x_5,x_6) \cap (x_1,x_2,x_5,x_7) \cap (x_1,x_2,x_7,x_8)  \cap
(x_1,x_3,x_4,x_6) \cap  (x_1,x_3,x_5,x_6) \cap (x_1,x_3,x_5,x_7) \cap  (x_1,x_3,x_6,x_8) \cap  (x_1,x_6,x_7,x_8) \cap (x_2,x_4,x_5,x_7) \cap
(x_2,x_4,x_6,x_8) \cap  (x_2,x_4,x_7,x_8) \cap (x_3,x_4,x_5,x_6) \cap  (x_3,x_4,x_6,x_8) \cap (x_3,x_4,x_7,x_8) \cap 
(x_4,x_5,x_6,x_7) \cap  (x_5,x_6,x_7,x_8) $ be an ideal in $R=\KK[x_1,\dots, x_8]$.
The minimal free resolution of its Alexander dual ideal is
$$\mathbb{L}_{\bullet}(I^\vee): \hskip 3mm \xymatrix{ 0 \ar[r] & R(-8)^{5} \ar[r] & R(-7)^{12}\oplus R(-6)^{4} \ar[r]& R(-5)^{28} \ar[r]& R(-4)^{18} \ar[r]& I^\vee
\ar[r]& 0}$$
and thus $I^\vee$ has two linear strands. The Lyubeznik table is:
$$\Lambda(R/I)=\begin{pmatrix}
 0 & 0 & 0 & 0& 0 \\
   & 0 & 0 & 7 & 0 \\
   &  & 0 & 0 & 0 \\
   &  &  & 0 & 7 \\
   &  &  &  & 1
\end{pmatrix}$$ 
The ring $R/I$ is $S_2$ (see \cite[Example 5.6]{Holmes}) but it is not Cohen-Macaulay because it has two local cohomology modules different from zero. It is not even generalized Cohen-Macaulay. This ring is also $CM_2$ so it satisfies the properties shown in Corollary \ref{cor36}.
\end{example}

\vskip 2mm

\begin{example}\label{ex1}
Let
$I=(x_1,x_2,x_3,x_4,x_5) \cap(x_1,x_2,x_3,y_4,y_5) \cap (y_1,y_2,y_3,y_4,y_5)$  be an ideal in $R=\KK[x_1,\dots, x_5, y_1,\dots,y_5]$.
The minimal free resolution of its Alexander dual ideal is
$$\mathbb{L}_{\bullet}(I^\vee): \hskip 3mm \xymatrix{ 0 \ar[r] & R(-7)\oplus R(-8) \ar[r]& R(-5)^3 \ar[r]& I^\vee
\ar[r]& 0}$$
and thus $I^\vee$ has three linear strands. The Lyubeznik table is:
$$\Lambda(R/I)=\begin{pmatrix}
0 &0 & 0 & 0 & 0& 0 \\
  &0 & 0 & 0 & 0& 0 \\
 &  & 0 & 1& 0 & 0 \\
 &  &  & 0 & 1 & 0 \\
 &  &  &  & 0 & 0 \\
 &  &  &  &  & 3
\end{pmatrix}$$ 
The ideal $I$ can be interpreted as the edge ideal of a graph $G(3,2)$ obtained from a Cohen-Macaulay bipartite graph $G$. Then, the ring $R/I$ is $CM_4$ by using \cite[Theorem 4.5]{HSYZ18}. The ring $R/J$ with $J=I \cap (x_1,x_2,x_3,x_4,y_1,y_2,y_3,y_4)$ is not equidimensional but it is still $CM_4$. Despite the fact that the minimal free resolution of its Alexander dual ideal is
$$\mathbb{L}_{\bullet}(J^\vee): \hskip 3mm \xymatrix{ 0 \ar[r] & R(-10) \ar[r] & R(-7)\oplus R(-8) \ar[r]\oplus R(-9)^2& R(-5)^3 \oplus R(-8) \ar[r]& J^\vee
\ar[r]& 0}$$ we have $\Lambda(R/I)=\Lambda(R/J)$.

\end{example}

\section{Squarefree initial ideals}

Let $R=\KK[x_1,\dots, x_n]$ be a polynomial ring with coefficients in a field $\KK$. Assume that $R$ is equipped with a $\ZZ^m$-graded structure such that $\deg(x_i)\in \ZZ^m_{\geq 0}$.  It has been know for a while that some homological invariants behave well with respect to Gr\"obner deformations. In \cite[Theorem 1.3]{CV20}, Conca and Varbaro  proved that  for a  $\ZZ^m$-graded ideal $I \subseteq R$ such that the initial ideal ${\rm in}(I)$ with respect to some term order  is squarefree, then $$\dim_{\KK} H^i_{\fM}(R/I)_\alpha = \dim_{\KK} H^i_{\fM}(R/{\rm in}(I))_\alpha$$ for all $i\in \ZZ_{\geq 0}$ and all $\alpha\in \ZZ^m$. Therefore, extremal Betti numbers, depth and Castelnuovo-Mumford regularity of $R/I$ and $R/{\rm in}(I)$ coincide.
Classes of ideals satisfying this condition are ASL ideals, Cartwright-Sturmfels ideals and Knutson ideals (see  \cite{CV20} for details).

\vskip 2mm
For our purposes we point out the following result:

\begin{proposition} \cite[Corollary 2.11]{CV20} 
Let $R=\KK[x_1,\dots, x_n]$ be a polynomial ring over a field.  Let $I \subseteq R$ be a pure homogeneous  ideal of codimension $\geq 2$ such that the initial ideal ${\rm in}(I)$ with respect to some term order  is squarefree. Then:

\begin{itemize}
\item[(i)] $R/I$ is $S_r$, $r\geq 2$, if and only if $R/ {\rm in}(I)$ is $S_r$.
\item[(ii)] $R/I$ is $CM_r$ if and only if $R/ {\rm in}(I)$ is $CM_r$.
\end{itemize}

\end{proposition}

\vskip 2mm

It has been proved by Alan\'is-L\'opez, N\'u\~nez-Betancourt and Ram\'irez-Moreno \cite{ANR22} that the graphs $\Gamma_t(R/I)$, and consequently some Lyubeznik numbers, also behave well with respect to Gr\"obner deformations.

\begin{theorem} \cite[Theorem 3.4]{ANR22}
Let $R=\KK[x_1,\dots, x_n]$ be a polynomial ring over a field.  Let $I \subseteq R$ be a pure homogeneous  ideal of codimension $\geq 2$ such that the initial ideal ${\rm in}(I)$ with respect to some term order  is squarefree. Then,
$$\# \Gamma_t(R/I) = \# \Gamma_t(R/ {\rm in}(I)).$$
\end{theorem}

\vskip 2mm

\begin{corollary}
Let $R=\KK[x_1,\dots, x_n]$ be a polynomial ring over a field.  Let $I \subseteq R$ be a pure homogeneous  ideal of codimension $\geq 2$ such that the initial ideal ${\rm in}(I)$ with respect to some term order  is squarefree. Then,
$$\lambda_{d,d}(R/I) = \lambda_{d,d}(R/ {\rm in}(I)) , \hskip 2mm  \lambda_{0,1}(R/I) = \lambda_{0,1}(R/ {\rm in}(I))  \hskip 2mm {\rm and} \hskip 2mm  \lambda_{1,2}(R/I) = \lambda_{1,2}(R/ {\rm in}(I)).$$
\end{corollary}

\vskip 2mm
In positive characteristic, Nadi and Varbaro \cite{NV19} proved the following inequality between the Lyubeznik numbers of $R/I$ and those of $R/{\rm in}(I)$.

\begin{proposition} \cite[Corollary 2.5]{NV19}  \label{NV}
Let $R=\KK[x_1,\dots, x_n]$ be a polynomial ring over a field of positive characteristic.  Let $I \subseteq R$ be an homogeneous  ideal such that the initial ideal ${\rm in}(I)$ with respect to some term order  is a squarefree monomial ideal. Then $\lambda_{p,i}(R/I) \leq \lambda_{p,i}(R/{\rm in}(I))$.
%
%
%

\end{proposition}

%

%
%

Combining this result with Theorems \ref{thm1} and \ref{thm2} we obtain the following:

\vskip 2mm

\begin{corollary}
Let $R=\KK[x_1,\dots, x_n]$ be a polynomial ring over a field of positive characteristic.  Let $I \subseteq R$ be an homogeneous  ideal such that the initial ideal ${\rm in}(I)$ with respect to some term order  is a squarefree monomial ideal. Then:

\begin{itemize}

\item If $R/{\rm in}(I)$ is $S_r$ with $r\geq 2$ then the Lyubeznik table of $R/I$  satisfies

\vskip 2mm

$\lambda_{i,i}(R/I)= \lambda_{i,i+1}(R/I)= \cdots =\lambda_{i,i+(r-1)}(R/I)=0,$ for $ i\in\{0,\dots, d-1\}.$

\vskip 2mm

\item If  $R/{\rm in}(I)$ is $CM_r$ then the Lyubeznik table of $R/I$  satisfies

\vskip 2mm

$\lambda_{p,i}(R/I)= 0 , \hskip 2mm \forall p\geq r $ and $i\in\{0,\dots, d-1\}. $

\end{itemize}
\end{corollary}

It is quite common that the Lyubeznik table of a monomial ideal is trivial and thus the following easy consequence becomes relevant.

\begin{corollary}
Let $R=\KK[x_1,\dots, x_n]$ be a polynomial ring over a field of positive characteristic.  Let $I \subseteq R$ be an homogeneous  ideal such that the initial ideal ${\rm in}(I)$ with respect to some term order  is a squarefree monomial ideal. If the Lyubeznik table of $R/{\rm in}(I)$ is trivial then the Lyubeznik table of $R/I$ is trivial as well.
\end{corollary}

\vskip 2mm

Using these results, we can compute the Lyubeznik table of the following examples:

\begin{example}\label{ex2} \cite[Example 3.2]{CV20}
Let $R=\KK[x_1,\dots, x_5]$ be a polynomial ring over a field of positive characteristic. Let $I$ be the homogeneous ideal given by the $2\times 2$-minors of the matrix 
$$\begin{pmatrix}
  x_4^2+x_5^a & x_3 & x_2  \\
   x_1& x_4^2 & x_3^b-x_2 
\end{pmatrix}$$ with $\deg(x_4) = a, \deg(x_1) = \deg(x_3) = 1, \deg(x_2) = b$ and $\deg(x_5) = 2$. 
On the other hand, 
$${\rm in}(I) = (x_1x_3, x_1x_2, x_2x_3)$$ where we consider the lex term order and thus the Lyubeznik table of $R/I$ is trivial in any characteristic.
\end{example}

Binomial edge ideals satisfy  that their generic initial ideals are squarefree \cite[Theorem 2.1]{CDNG18}.


\begin{example}\label{ex3}
Let $R=\KK[x_1,\dots, x_6, y_1, \dots, y_6]$ be a polynomial ring over a field of positive characteristic. Let $J_{C_6}\subseteq R$ be the binomial edge ideal associated to the $6$-cycle $C_6$ and ${\rm gin}(J_{C_6})$ its generic initial ideal. Namely, we have: 

\vskip 2mm

$J_{C_6}= (x_1y_2-x_2y_1,x_1y_6-x_6y_1,x_2y_3-x_3y_2,-x_3y_4+x_4y_3,x_4y_5-x_5y_4,x_5y_6-x_6y_5) $
\begin{align*}
{\rm gin}(J_{C_6}) = &(x_5x_6,x_4x_5,x_3x_4,x_2x_3,x_1x_6,x_1x_2,x_4x_6y_5,x_3x_5y_4,x_2x_6y_1,x_2x_4y_3,x_1x_5y_6,x_1x_3y_2, \\
& x_3x_6y_4y_5,x_3x_6y_1y_2,x_2x_5y_3y_4,x_2x_5y_1y_6,x_1x_4y_5y_6,x_1x_4y_2y_3,x_4x_6y_1y_2y_3,x_3x_5y_1y_2y_6,\\
& x_2x_6y_3y_4y_5,x_2x_4y_1y_5y_6,x_1x_5y_2y_3y_4,x_1x_3y_4y_5y_6) 
\end{align*}

\vskip 2mm

The Lyubeznik table of $R/ {\rm gin}(J_{C_6})$ is trivial in any characteristic and thus the Lyubeznik table of $R/ J_{C_6}$ is trivial as well. 

\end{example}




\bibliographystyle{alpha}
\bibliography{References}

\end{document}